\documentclass[letterpaper, 10 pt, conference]{ieeeconf}  % Comment this line out if you need a4paper

\IEEEoverridecommandlockouts                              % This command is only needed if 
\usepackage{todonotes}
\usepackage{latexsym,amssymb,amsmath, amsbsy,amsopn, amstext}
\usepackage{graphicx,epsfig,tikz,pgf,hyperref,cancel,color,float,ifpdf}
\usepackage{extarrows,mathrsfs}
\usepackage{amsmath,  amssymb,stfloats}
\usepackage{tikz,subcaption}
\usepackage[]{caption}
\usetikzlibrary{automata, positioning, arrows}
\usepackage[ruled,vlined,linesnumbered]{algorithm2e}

\usepackage{bm}

\usepackage{etoolbox}
\makeatletter
\patchcmd{\@makecaption}
  {\scshape}
  {}
  {}
  {}
\makeatother
\usepackage{makecell}

\makeatletter
\renewcommand\normalsize{%
\@setfontsize\normalsize\@xpt\@xiipt
\abovedisplayskip 8\p@ \@plus2\p@ \@minus5\p@
\abovedisplayshortskip \z@ \@plus3\p@
\belowdisplayshortskip 8\p@ \@plus3\p@ \@minus3\p@
\belowdisplayskip \abovedisplayskip
\let\@listi\@listI}
\makeatother

\graphicspath{{./figures/}}

\newtheorem{definition}{\bfseries Definition}%[section]
\newtheorem{proposition}{\bfseries Proposition}%[section]
\newtheorem{theorem}{\bfseries Theorem}
\newtheorem{corollary}{\bfseries Corollary}%[section]
\newtheorem{lemma}{\bfseries Lemma}%[section]

\newtheorem{remark}{\bfseries Remark}

\newtheorem{problem}{\bfseries Problem}

\def\x{\bm{x}}

\def\u{\bm{u}}

\def\f{\bm{f}}

\newcommand{\z}{\bm{z}}
\newcommand{\cc}{\mathbf{c}}
\newcommand{\G}{\mathbf{G}}
\newcommand{\A}{\mathbf{A}}
\newcommand{\bb}{\mathbf{b}}
\newcommand{\R}{\mathbb{R}}
\newcommand{\IR}{\mathbb{IR}}

\newcommand{\Zc}{\mathcal{Z}}

\newcommand{\mat}[1]{\begin{bmatrix} #1 \end{bmatrix}}

\allowdisplaybreaks[4]

\usepackage{xcolor}
\newif\ifdraft

\drafttrue

\title{\LARGE \bf
Safety Verification of Neural Feedback Systems Based on Constrained Zonotopes}

\author{Yuhao Zhang and Xiangru Xu
\thanks{Yuhao Zhang and Xiangru Xu are with the Department of Mechanical Engineering, University of Wisconsin-Madison,
        Madison, WI 53706, USA. Email: 
        {\tt\small \{yuhao.zhang2,xiangru.xu\}@wisc.edu}.}%
}

\begin{document}
\maketitle
\begin{abstract}
Artificial neural networks have recently been utilized in many feedback control systems and introduced new challenges regarding the safety of such systems. This paper considers the safe verification problem for a dynamical system with a given feedforward neural network as the feedback controller by using a constrained zonotope-based approach. 
A novel set-based method is proposed to compute both exact and over-approximated reachable sets for neural feedback systems with linear models, and linear program-based sufficient conditions are presented to verify whether the trajectories of such a system can avoid unsafe regions represented as constrained zonotopes. 
The results are also extended to neural feedback systems with nonlinear models. The computational efficiency and accuracy of the proposed method are demonstrated by two numerical examples where a comparison with state-of-the-art methods is also provided.

\end{abstract}

\section{Introduction}\label{sec:intro}

%\XX{overview figure}

With the universal approximation theorem \cite{scarselli1998universal}, artificial neural networks (ANNs) have become an effective and powerful tool for many complex applications such as image segmentation \cite{pal1993review}, natural language translation \cite{hinton2012deep}, and autonomous driving \cite{shengbo2019key}. Despite its success, many types of ANNs have been shown to lack robustness to small input perturbations \cite{goodfellow2014explaining}. Therefore, for control systems with ANN components, it's important to formally verify their safety properties before real implementations.
%before applying ANNs to control systems that are safety-critical, it's important to first verify the safety properties of such closed-loop systems. % which are called \emph{neural feedback systems} \cite{dutta2019reachability}. 

%\begin{figure}[!ht]
%    \centering
%        \includegraphics[width=0.40\textwidth]{neural-feedback-loop.jpg}
%     \caption{Neural feedback system.}
%      \label{fig:feedback-loop}
%\end{figure}

Due to the highly non-convex and nonlinear natures, the reachability analysis and safety verification of ANNs are notoriously difficult: it is shown that even verifying simple properties about ANNs is an NP-complete problem \cite{katz2017reluplex}. Recently, analyzing the safety and robustness of ANNs has attracted attention from the machine learning and formal methods research communities.  By exploiting the piecewise-linear nature of the Rectified Linear Unit (ReLU) activation function, the analysis of ANNs can be reduced to a constraint satisfaction problem that can be solved by mixed-integer linear programming \cite{dutta2017output} or satisfiability modulo theory techniques \cite{katz2017reluplex}. Methods that rely on different set representations, such as polytopes \cite{tran2019parallelizable,vincent2021reachable}, zonotopes \cite{singh2018fast}, constrained zonotopes \cite{chung2021constrained}, and star sets \cite{tran2019star} have been proposed to analyze the reachability of ANNs; however, these works above only focus on analyzing ANNs in isolation.

Several recent works propose methods to compute forward reachable sets for neural feedback systems \cite{dutta2019reachability,huang2019reachnn,xiang2020reachable,ivanov2019verisig,hu2020reach,everett2021reachability}. A reachable set over-approximation method is proposed in \cite{hu2020reach} based on quadratic constraints and semi-definite programming, and the method is extended in \cite{everett2021reachability} by leveraging linear programming (LP) and set partitioning; however, these relaxation-based methods  are unable to compute the exact reachable set of the neural feedback system. Learning-based methods are also proposed to approximate reachable sets for neural feedback systems \cite{chakrabarty2020active,devonport2020data}; however, these methods can only provide a probabilistic guarantee on the correctness of the approximated reachable sets.%, and a method to compute exact reachable sets of neural feedback systems is still lacking. %Constrained zonotope was introduced in \cite{scott2016constrained} as a novel set representation and proved to be equivalent to convex polytopes. %It has been shown advantages in calculating many classical set operations \cite{raghuraman2022set}. 

%The goal of this paper is to provide efficient and accuracy methods to guarantee the systems with feed-forward neural network (FNN) controllers will not enter unsafe regions using a reachability analysis approach. 
%Figure \ref{fig:framework} shows an overview framework of our proposed reachability analysis for neural feedback systems. 
In this work, we leverage the properties of constrained zonotopes and deploy set-based analysis techniques to verify the safety of neural feedback systems, which are dynamical systems with a given ReLU-activated feedforward neural network (FNN) as the feedback controller. 
%The FNN considered in this work has the Rectified Linear Unit (ReLU) activation function. 
The contributions of this work are at least threefold: (i) Based on the output reachability analysis of FNNs, two novel methods are proposed to compute the exact and over-approximated reachable sets of neural feedback systems;  
% To the best of our knowledge, Theorem \ref{thm:sum} is the first result that can compute the exact reachable sets for linear discrete-time systems with FNN controllers. 
%The proposed over-approximated method concerns the trade-off between computation efficiency and approximation accuracy and has a better performance comparing with other existing over-approximated algorithms. 
(ii) LP-based sufficient conditions are proposed to verify the avoidance of unsafe sets for neural feedback systems; (iii) The proposed reachability analysis and safety verification methods are extended to neural feedback systems with nonlinear models. An overview of the proposed framework is illustrated in Figure \ref{fig:framework}. 

The remainder of the paper is laid out as follows:
Section \ref{sec:pre} introduces preliminaries on constrained zonotopes and interval arithmetics and presents the problem statement. Section \ref{sec:fnn} introduces the constrained zonotope-based output analysis of FNNs in isolation. Section \ref{sec:verify} presents two reachable set computation methods for linear discrete-time systems with FNN controllers as well as two corresponding sufficient conditions to certify the safety of the neural feedback systems. Section \ref{sec:nonlin} extends the reachability analysis and safety verification to systems with nonlinear models. Two numerical examples are shown in Section \ref{sec:sim} before the paper is concluded in Section \ref{sec:concl}.
\begin{figure*}[!ht]
    \centering
        \includegraphics[width=0.75\textwidth]{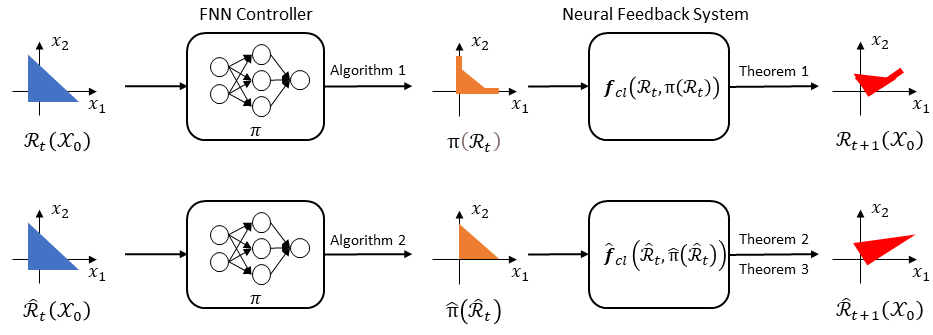}
    \caption{The top (resp. bottom) flowchart illustrates the exact (resp. over-approximated) constrained zonotope-based reachability  analysis, where $\mathcal{R}_t(\mathcal{X}_0)$ (resp. $\hat{\mathcal{R}}_t(\mathcal{X}_0)$) denotes the exact (resp.  over-approximated) reachable set at time $t$ from initial set $\mathcal{X}_0$, $\pi(\mathcal{R}_t)$ (resp.  $\hat{\pi}(\hat{\mathcal{R}}_t)$) denotes the exact  (resp. over-approximated) output set of the FNN controller, and $\f_{cl}$ (resp. $\hat{\f}_{cl}$) is the exact (resp. over-approximated) reachability mapping for the closed-loop neural feedback system.}
    %  \caption{Framework of constrained zonotope-based reachability analysis for neural feedback systems. The top flowchart is the exact reachable set analysis while the flowchart below is the over-approximated reachable set analysis. $\mathcal{R}_t(\mathcal{X}_0)$ and $\hat{\mathcal{R}}_t(\mathcal{X}_0)$ denote the exact and the over-approximated reachable set at time $t$ from initial set $\mathcal{X}_0$ respectively. $\pi(\mathcal{R}_t)$ and $\hat{\pi}(\hat{\mathcal{R}}_t)$ are the exact and over-approximated output set of the FNN controller respectively. $\f_{cl}$ and $\hat{\f}_{cl}$ are the exact and over-approximated reachability methods for the closed-loop neural feedback system respectively.}
      \label{fig:framework}
      %\vspace{-0.4cm}
\end{figure*}

% The remainder of the paper is laid out as follows:
% Section \ref{sec:pre} introduces preliminaries on constrained zonotope, interval arithmetic and reachability, and presents the problem statement. Section \ref{sec:fnn} introduces the set-based output analysis of FNNs in isolation. Section \ref{sec:verify} presents two reachable set computation methods for linear discrete-time systems with FNN controllers as well as two corresponding sufficient conditions to certify the safety of the neural feedback systems. Section \ref{sec:nonlin} extends the reachability analysis and safety verification to systems with nonlinear dynamics. Two numerical examples are shown in Section \ref{sec:sim} before the paper is concluded in Section \ref{sec:concl}.

\section{Preliminaries \& Problem Statement}\label{sec:pre}

\subsection{Constrained Zonotope}\label{sec:zono}

\begin{definition} \cite{scott2016constrained}
A set $\Zc \subset \mathbb{R}^{n}$ is a constrained zonotope if there exists $(\mathbf{c}, \mathbf{G}, \mathbf{A}, \mathbf{b}) \in  \mathbb{R}^{n} \times \mathbb{R}^{n \times n_{G}}\times \mathbb{R}^{n_{A} \times n_{G}} \times \mathbb{R}^{n_{A}}$ such that
$
\Zc=\left\{\mathbf{G} \bm{\xi}+\mathbf{c} \;|\; \|\bm{\xi}\|_{\infty} \leq 1, \mathbf{A} \bm{\xi}=\mathbf{b}\right\}.
$
\end{definition}

Denote the constrained zonotope defined by $(\mathbf{c}, \mathbf{G}, \mathbf{A}, \mathbf{b})$ as $CZ\{\mathbf{c}, \mathbf{G}, \mathbf{A}, \mathbf{b}\}$. Denote the unit hypercube by $B_{\infty}$ and define $B_{\infty}(\A,\bb) = \{\bm{\xi}\in B_{\infty}\;|\;\A\bm{\xi} = \bb\}$. It's proven that a constrained zonotope is equivalent to a convex polytope \cite[Theorem 1]{scott2016constrained}. 
Constrained zonotopes are closed under linear map, Minkowski sum and intersection as shown in the following result.

\begin{lemma}\cite[Proposition  1]{scott2016constrained}\label{lemma:set-op}
For every $\mathbf{R} \in \mathbb{R}^{k \times n}$, $\Zc=CZ\left\{ \mathbf{c}_{z},\mathbf{G}_{z},\mathbf{A}_{z},\mathbf{b}_{z}\right\} \subset \mathbb{R}^{n}$, and $\mathcal W= CZ\{ \mathbf{c}_{w},\mathbf{G}_{w},\mathbf{A}_{w}, $ $\mathbf{b}_{w}\} \subset \mathbb{R}^{n}$, 
%$\mathcal Y= CZ\{ \mathbf{c}_{y}, \mathbf{G}_{y},\mathbf{A}_{y}, \mathbf{b}_{y}\} \subset \mathbb{R}^{k}$, 
the following three identities hold:
$$
\begin{aligned}
\mathbf{R}\Zc &=CZ\left\{\mathbf{R} \mathbf{c}_{z}, \mathbf{R G}_{z}, \mathbf{A}_{z}, \mathbf{b}_{z}\right\}, \\
\Zc\oplus \mathcal W &=CZ\left\{ \mathbf{c}_{z}+\mathbf{c}_{w},[\mathbf{G}_{z}\; \mathbf{G}_{w}],\mat{
\mathbf{A}_{z} & \mathbf{0} \\
\mathbf{0} & \mathbf{A}_{w}
},\mat{
\mathbf{b}_{z} \\
\mathbf{b}_{w}
}\right\}, \\
\Zc \cap \mathcal W &=CZ\left\{ \mathbf{c}_{z},[\mathbf{G}_{z}\;\mathbf{0}],\mat{
\mathbf{A}_{z} & \mathbf{0} \\
\mathbf{0} & \mathbf{A}_{w} \\
\mathbf{G}_{z} & -\mathbf{G}_{w}
},\mat{
\mathbf{b}_{z} \\
\mathbf{b}_{w} \\
\mathbf{c}_{w}-\mathbf{c}_{z}
}\right\},
\end{aligned}
$$
where $\oplus$ denotes the Minkowski sum.
\end{lemma}
%\XX{No need to introduce $\mathcal Y$ in this definition. Please check. Or use the generalized intersection.  }

%Noting that constrained zonotopes can be empty, and that 
Checking the emptiness of a constrained zonotope requires the solution of a LP.
\begin{lemma}\label{lemma:empty} \cite[Proposition 2]{scott2016constrained}
For every $\mathcal{Z} = CZ\{\cc,\G,$ $\A,\bb\}\subset \mathbb{R}^n$, 
$%\label{equ:empty}
    \mathcal{Z} \not = \emptyset$ iff $\min\{||\bm{\xi}||_{\infty} \;|\; \A\bm{\xi} = \bb\} \leq 1.
$
\end{lemma}

The intersection of a constrained zonotope and a halfspace is still a constrained zonotope as shown in the following result.
%The following lemma provides a formula for computing the 

\begin{lemma}\label{lemma:halfspace} \cite[Theorem 1]{raghuraman2022set}
If the constrained zonotope $\Zc=CZ\{\mathbf{c}, \mathbf{G}, \mathbf{A}, \mathbf{b}\} \subset \mathbb{R}^{n}$ intersects the hyperplane $\mathcal H=\left\{\x \in \mathbb{R}^{n} \mid \mathbf{h}^{T} \x=f\right\}$ corresponding to the halfspace $\mathcal{H}_{-}=\left\{\x \in \mathbb{R}^{n} \mid \mathbf{h}^{T} \x \leq f\right\}$, then the intersection $\Zc_{h}=\Zc \cap \mathcal{H}_{-}$is a constrained zonotope 
\begin{align*}
\Zc_{h}=\left\{\mathbf{c} , \mat{\G & \mathbf{0}},\mat{
\mathbf{A} & \mathbf{0} \\
\mathbf{h}^{T} \mathbf{G} & \frac{d_{m}}{2}
},\mat{
\mathbf{b} \\
f-\mathbf{h}^{T} \mathbf{c}-\frac{d_{m}}{2}
}\right\}
\end{align*}
% $$\!
% \Zc_{h}=\left\{\mathbf{c} , \mat{\G & \mathbf{0}},\mat{
% \mathbf{A} & \mathbf{0} \\
% \mathbf{h}^{T} \mathbf{G} & \frac{d_{m}}{2}
% },\mat{
% \mathbf{b} \\
% f-\mathbf{h}^{T} \mathbf{c}-\frac{d_{m}}{2}
% }\right\}
% $$
%and $d_{m}=f-\mathbf{h}^{T} \mathbf{c}+\sum_{i=1}^{n_{G}}\left|\mathbf{h}^{T} \mathbf{G}[:,i]\right|$, where 
where $d_{m}=f-\mathbf{h}^{T} \mathbf{c}+\sum_{i=1}^{n_{G}}\left|\mathbf{h}^{T} \mathbf{G}[:,i]\right|$ and $\mathbf{G}[:,i]$ is the $i$-th column of matrix $\mathbf{G}$.
%\XX{I did not find theorem 1 of \cite{rego2021set}.}
\end{lemma}

In the following, we denote $\mathbf{e}_i$ as the $i$-th canonical vector, 
$\mathcal{H}^i = \{\x \in \mathbb{R}^{n} \mid \mathbf{e}_i^T\x = 0\}$, 
% $\mathcal{H}^i = \left\{\x \in \mathbb{R}^{n} \mid \x[i] = 0 \right\} = \{\x \in \mathbb{R}^{n} \mid \mathbf{e}_i^T\x = 0\}$ 
$\mathcal{H}^i_- =  \{\x \in \mathbb{R}^{n} \mid \mathbf{e}_i^T\x \leq 0\}$, and $\mathcal{H}^i_+ =  \{\x \in \mathbb{R}^{n} \mid \mathbf{e}_i^T\x \geq 0\}$ for $i=1,\dots,n$.

\subsection{Interval Arithmetic}

A real interval $[a]=[\underline{a},\bar a]$ is a subset of $\R$. 
%The midpoint and radius of $[a]$ are defined by $mid([a]) \triangleq \frac{1}{2}(\underline{a}+\bar a)$ and $rad([a]) \triangleq \frac{1}{2}(\bar a - \underline{a})$, respectively.
Denote $\IR$, $\IR^n$ and $\IR^{n\times m}$ as the set of all real intervals of $\R$, all $n$-dimensional real interval vectors and all $n\times m$ real interval matrices, respectively. 
Real arithmetic operations on $\R$ can be extended to $\IR$ as follows: for $\circ \in\{+,-,*,\div\}$, 
$[a]\circ[b]=\{\inf_{x\in[a],y\in[b]}x\circ y, \sup_{x\in[a],y\in[b]}x\circ y\}$. The classical operations for real vectors and real matrices can be directly extended to interval vectors and interval matrices \cite{jaulin2001interval,moore2009introduction}.

For a bounded set $\mathcal{X} \subset \mathbb{R}^n$, denote $\Box \mathcal{X}$ as the interval hull of $\mathcal{X}$. The interval hull $\Box \mathcal{Z}$ of a constrained zonotope $\mathcal{Z}$ can be computed using the LP proposed in \cite{scott2016constrained,rego2020guaranteed}. 
Denote $\lhd ([\mathbf{J}],\mathcal{Z})$ as the operation of computing a constrained zonotopic enclosure of the
product of an interval matrix $[\mathbf{J}]$ and a constrained zonotope $\mathcal{Z}$ using \cite[Theorem 1]{rego2020guaranteed}.

\subsection{Problem Statement}

We consider the following discrete-time control system:
\begin{equation}\label{dt-sys}
    \x{(t+1)} = \f( \x(t)) + B_d \u(t)% + w(t)
\end{equation}
where $\x(t)\in \mathcal{X} \subseteq \mathbb{R}^n$ is the state, $\u(t)\in \mathbb{R}^m$ is the control input,  %and $w(t) \in \mathcal{W}\subseteq \mathbb{R}^n$ is the disturbance. 
%$\f:\mathcal{X}\rightarrow \mathbb{R}^n$ and $\g:\mathcal{X}\rightarrow \mathbb{R}^{n\times m}$ are vector-valued
%and matrix-valued functions, respectively. We assume that $\f$ and $\g$ are twice continuously differentiable. 
%$A_d \in \mathbb{R}^{n \times n}$ and $B_d \in \mathbb{R}^{n \times m}$ are given state and input matrices, respectively.
$\f:\mathcal{X}\rightarrow \mathbb{R}^n$ is a twice continuously differentiable vector-valued function (i.e., $\f$ is of class $\mathcal{C}^2$), and $B_d \in \mathbb{R}^{n \times m}$ is a given input matrix. Given an initial state $\x(0)$ and a control sequence $\u = \u(0), \u(1), \dots$, the state trajectory of system \eqref{dt-sys} is denoted as $\x = \x(0),\x(1),\dots$.

Assume that the controller in \eqref{dt-sys} is a state-feedback controller $\u(t) = \pi(\x(t))$ that is parameterized by an $\ell$-layer FNN  with the Rectified Linear Unit (ReLU) activation function. 
%, i.e., $ReLU(\x) = \max\{0,\x\}$
%The controller $\pi$ in the closed-loop system \eqref{close-sys} is assumed to be a $\ell$-layer FNN with ReLU activation functions. 
Letting $\x^{(0)} = \x(t)$ and using the notation from \cite{chung2021constrained}, for each layer $k=1,\dots,\ell-1$, the neuron of the FNN is given by
\begin{equation}\label{equ:NN}
\x^{(k)}=ReLU\left(\mathcal{L}\left(\x^{(k-1)}, \mathbf W^{(k-1)}, \mathbf v^{(k-1)}\right)\right)    
\end{equation}
where $\mathbf W^{(k-1)}$ is the $k$-th layer weight matrix, $\mathbf v^{(k-1)}$ is the $k$-th layer bias vector, $\mathcal{L}(\x, \mathbf W, \mathbf v) =\mathbf W\x+\mathbf v$, and 
$ReLU(\x) =\max \{0, \x\}$.
% $$
% \begin{aligned}
% \mathcal{L}(\x, \mathbf W, \mathbf v) &=\mathbf W\x+\mathbf v, \\
% ReLU(\x) &=\max \{0, \x\}.
% \end{aligned}
% $$
In the last layer, only the linear operation is applied, i.e., $\pi(\x(t)) = \x^{(\ell)} = \mathcal{L}(\x^{(\ell-1)},\mathbf W^{(\ell-1)},\mathbf v^{(\ell-1)})$.

The closed-loop system with dynamics \eqref{dt-sys} and the controller $\u(t) = \pi(\x(t))$ becomes:
\begin{equation}\label{close-sys}
    \x{(t+1)} = \f( \x(t)) + B_d \pi(\x(t)).
\end{equation}

Given a set of initial states $\mathcal{X}_0\subseteq \mathbb{R}^n$, the (forward) reachable set of closed-loop system \eqref{close-sys} at time $t$ from the set $\mathcal{X}_0$ is denoted as $\mathcal{R}_t(\mathcal{X}_0)\triangleq\{\x(t)\in \mathbb{R}^n | \x(0) \in \mathcal{X}_0, \x{(k+1)} = \f( \x(k)) + B_d \pi(\x(k)), k = 0,1,\dots, t-1\}$, or simply $\mathcal{R}_t$ when $\mathcal{X}_0$ is clear from context. Denote an over-approximation of the set $\mathcal{R}_t(\mathcal{X}_0)$ as $\hat{\mathcal{R}}_t(\mathcal{X}_0)$.

%\subsection{Problem Statement}
In this paper, we investigate the following problems in which the initial set and the unsafe sets are all assumed to be in the form of constrained zonotopes.

\begin{problem}
Given an initial set $\mathcal{X}_0$ that is represented as a constrained zonotope, the parameters for the FNN controller $\pi$ and a time horizon $T\in \mathbb{Z}_{>0}$, compute the exact reachable set $\mathcal{R}_t(\mathcal{X}_0)$ or an over-approximated reachable set $\hat{\mathcal{R}}_t(\mathcal{X}_0)$ for the closed-loop system \eqref{close-sys} where $t=1,2,\dots,T$.
\end{problem}

\begin{problem}
Given unsafe sets $\{\mathcal{O}_1,\mathcal{O}_2,\dots,\mathcal{O}_N\}$ where $\mathcal{O}_i(1\leq i\leq N)$ is represented as a constrained zonotope, verify whether the state trajectory of the closed-loop system \eqref{close-sys} can avoid the unsafe regions for $t=1,2,\dots,T$.%in $T$ time horizon.
\end{problem}

%\XX{a remark about reach-avoid problem using Tedrake's paper? CZ's containment problem?}

%Give the initial set is $\mathcal{X}_0$ and the unsafe sets $\{\mathcal{O}_1,\mathcal{O}_2,\dots,\mathcal{O}_N\}$, we want to verify whether the sequences of the closed-loop system will enter the unsafe set or not. Since the exact reachable set is hard to get, we will compute an over-approximated reachable set $\tilde{\mathcal{R}}_t \supseteq \mathcal{R}_t$ at each time step $t$ and check the emptiness of the intersection between the reachable sets and the unsafe sets.

\section{Output Analysis of FNNs Based On Constrained Zonotopes}\label{sec:fnn}

% The controller $\pi$ in the closed-loop system \eqref{close-sys} is assumed to be a $\ell$-layer FNN with ReLU activation functions. Let $\x^{(0)} = \x(t)$. Using notation from \cite{chung2021constrained}, for each layer $k=1,\dots,\ell-1$, the neuron of the neural network
% is given by
% \begin{equation}\label{equ:NN}
% \x^{(k)}=ReLU\left(\mathcal{L}\left(\x^{(k-1)}, \mathbf W^{(k-1)}, \mathbf v^{(k-1)}\right)\right)    
% \end{equation}
% where
% $$
% \begin{aligned}
% \mathcal{L}(\x, \mathbf W, \mathbf v) &=\mathbf W\x+\mathbf v, \\
% ReLU(\x) &=\max \{0, \x\}.
% \end{aligned}
% $$

% At last layer, we only apply the linear operation: $\pi(\x(t)) = \x^{(\ell)} = \mathcal{L}(\x^{(\ell-1)},\mathbf W^{(\ell-1)},\mathbf v^{(\ell-1)})$.

% Given an input set to an FNN in the form of constrained zonotopes, in this section, we will consider the problem of computing the output set of the FNN.

\subsection{Exact Output Analysis} \label{sec:nn}

%In this subsection, we will compute the exact reachable set $\mathcal{R}_t(\mathcal{X}_0)$ for a given FNN  shown in \eqref{equ:NN} and an input set $\mathcal{X}_0$ that is represented as a constrained zonotope.

In this subsection, we will compute the exact output set for a given FNN  shown in \eqref{equ:NN} with an input set  represented as a constrained zonotope.

From the definition of the FNN in \eqref{equ:NN}, one can see that the output set of an FNN can be derived layer by layer as the output of layer $k$ is the input of layer $k+1$, for $k=1,\dots,\ell-1$. 
Therefore, we will focus on finding the input-output relationship for one layer. 
From Lemma \ref{lemma:set-op}, if we pass an input represented as a constrained zonotope $\mathcal{Z} = CZ\{\cc,\G,\A,\bb\}$ through a linear layer, then we will obtain the output as $\mathcal{L}(\mathcal{Z}, \mathbf W, \mathbf v) = CZ\{\mathbf W\cc + \mathbf v,\mathbf W\G,\A,\bb\}$. Thus, the only difficulty remaining is to find the output when passing through the ReLU activation function. 
%We firstly consider the problem of computing the output set of a FNN $\pi$. For simplicity, we assume the input set of the FNN is given as a constrained zonotope $\mathcal{Z}_{in} = CZ\{\mathbf c_{in},\mathbf G_{in},\mathbf A_{in},\mathbf b_{in}\}$ and the activation functions are rectified linear unit (ReLU) functions, we want to find the output set $\pi(\mathcal{Z}_{in})$. 
In \cite{tran2019star}, an algorithm is proposed to compute the exact output set for a single neural network layer using the star sets representation. It can be shown that constrained zonotope is a special case of star sets \cite{althoff2021set}. Therefore, we can apply the algorithm in \cite{tran2019star} to compute the exact output set using constrained zonotopes. The details are summarized in Algorithm \ref{alg:exact} for which both the input set and the output set are in the form of constrained zonotopes.

%Since the ReLU activation function is applied to the neurons element-wise, we can treat it as taking the max operation for each neuron separately. For the $i$-th neuron in the $k$-th layer $\x^{(k)}[i]$, we firstly compute the range of $\x^{(k)}[i]$ in the input set. If the lower bound is above 0, then the max operation will not change the input set. If the upper bound is below 0, the 
%max operation will zero out the $i$-th dimension of the input set. And lastly, if 0 is inside the range of $\x^{(k)}[i]$, then, the input set will divided into two subsets by taking the intersections of the input set with the positive half-space and the negative half-space separately using Lemma \ref{lemma:halfspace}. The max operation will then zero out the $i$-th dimension of the negative intersection and the output set will be the union of two subsets. Algorithm 1 summarizes the detailed procedures mentioned above for computing the output set of a single layer.

Algorithm \ref{alg:exact} reveals that given a constrained zonotope as the input set $\mathcal{Z}$ to the FNN $\pi$, the exact output of the FNN can be represented as a union of constrained zonotopes: 
$
\pi(\mathcal{Z}) = \bigcup_{i=1}^{n_{\pi}} CZ \{\mathbf{c}_{i}, \mathbf{G}_{i}, \mathbf{A}_{i}, \mathbf{b}_{i}\},
$
where $n_{\pi}$ depends on the depth of the FNN $\pi$, the number of neurons and the output intersections with halfspaces.

\begin{algorithm}
\caption{(Adapted from Algorithm 3.1 in \cite{tran2019star}) Exact output analysis for one layer of FNN}\label{alg:exact}
\KwIn{weight matrix $\mathbf W$, bias vector $\mathbf v$, constrained zonotope input sets $\mathcal{Z} = \mathcal{Z}_1 \cup \mathcal{Z}_2 \cup \dots\cup \mathcal{Z}_{N_z}$}
\KwOut{exact output set $\mathcal{R}$}

\SetKwFunction{reachnn}{ReachNN}
\SetKwFunction{reachrelu}{ReachReLU}
\SetKwFunction{step}{StepReLU}
  
  \SetKwProg{Fn}{Function}{:}{}
  \Fn{$\mathcal{R}$ = \reachnn{$\mathcal{Z}$,$\mathbf W$,$\mathbf v$}}{
        $\mathcal{R} = \emptyset$\\
        \For{$h = 1:N_z$}{
        $\mathcal{I} = \mathbf W \mathcal{Z}_{h} + \mathbf v$\\
        $[lb\quad up]\leftarrow$ range of $\x$ in $\mathcal{I}$\\
        $map = find(lb<0)$\\
        \For {$i$ in $map$} {
        $\mathcal{I} = \step (\mathcal{I},i,lb[i],up[i])$
        }
        $\mathcal{R} =\mathcal{R}\cup \mathcal{I}$}
        \KwRet $\mathcal{R}$\
  }
  
   \SetKwProg{Fn}{Function}{:}{}
  \Fn{$\tilde{\mathcal{I}}$ = \step{${\mathcal{I}}$,$i$,$lb_i$,$up_i$}}{
        ${\mathcal{I}} = \mathcal{I}_1 \cup \mathcal{I}_2\cup\cdots\cup\mathcal{I}_{N_I} \subset \mathbb{R}^{n_I}$\\
        $\tilde{\mathcal{I}} = \emptyset$, $\mathbf E_i = [\mathbf e_1\; \cdots\; \mathbf e_{i-1}\; \mathbf{0}\; \mathbf e_{i+1}\; \cdots\; \mathbf e_{n_I}]$\\
        \For{$j=1:N_I$}{
        \If{$up_i \leq 0$}{
          $\hat{\mathcal{I}} = \mathbf E_i {\mathcal{I}}_j$
          }
        \If{$lb_i<0\; \&\; up_i >0$}{
          $\mathcal{I}_+ = {\mathcal{I}}_j\cap \mathcal{H}^i_+$\\
          $\mathcal{I}_- = {\mathcal{I}}_j\cap \mathcal{H}^i_-$\\
          $\hat{\mathcal{I}} = \mathcal{I}_+ \cup \mathbf E_i \mathcal{I}_-$
        }
        
        $\tilde{\mathcal{I}} = \tilde{\mathcal{I}} \cup \hat{\mathcal{I}}$
        }
        \KwRet $\tilde{\mathcal{I}}$\
  }

\end{algorithm}

\subsection{Over-approximation Output Analysis} \label{sec:approximate}

One major drawback of Algorithm \ref{alg:exact} is that in the worst scenario, the number of constrained zonotopes in the output set will grow exponentially with the number of layers and the number of neurons. Thus, it will dramatically increase the computation burden of output analysis for deep neural networks. In this subsection, we will introduce an algorithm that can over-approximate the output set of an FNN with one constrained zonotope. Similar to the over-approximation methods developed for star sets in \cite{tran2019star}, we construct the output set of each layer of the FNN as a constrained zonotope using a triangle over-approximation of the ReLU activation function for each neuron. As noted in \cite{tran2019star}, the star-based over-approximation algorithm is much less conservative than the zonotope-based \cite{singh2018fast} and abstract domain \cite{singh2019abstract} based approaches in approximating the ReLU function.

Figure \ref{fig:over} shows the main idea of ReLU function over-approximation. Given a range of the $i$-th neuron $\x[i]$ as $[lb_i,up_i]$, the output of the ReLU activation function can be divided into two parts (i.e., $I_1$ and $I_2$) that can be covered by the gray triangle area $\hat{I}$ (including the boundaries) which is the intersection of three halfspaces: $\bm y[i] - \x[i]\geq 0$, $\x[i] \geq 0$ and $(up_i-lb_i)\bm y[i] - up_i(\x[i]-lb_i) \leq 0$.  Using this convex relaxation, we  modify Algorithm \ref{alg:exact} and design  Algorithm \ref{alg:over} to compute the  over-approximated output set as a single constrained zonotope at each time step.

\begin{remark}
From Line 14 to Line 19 in Algorithm \ref{alg:over}, we know that if the lower bound of the neuron is negative and the upper bound of the neuron is positive, then the algorithm will introduce four new variables and three new equality constraints. In the worst case, if we have totally $M$ number of neurons in the FNN, the number of new variables in the over-approximated output set will be $4 M$ and the number of new constraints will be $3 M$. This could cause a computational burden issue for a deep FNN. However, it's possible to utilize the order reduction methods proposed in \cite{scott2016constrained,raghuraman2022set} to reduce the complexity of the approximated constrained zonotopes. %Nevertheless, the efficiency of the proposed algorithms will be shown in Section \ref{sec:sim}.
\end{remark}

\begin{figure}[!th]
    \centering
        \includegraphics[width=0.25\textwidth]{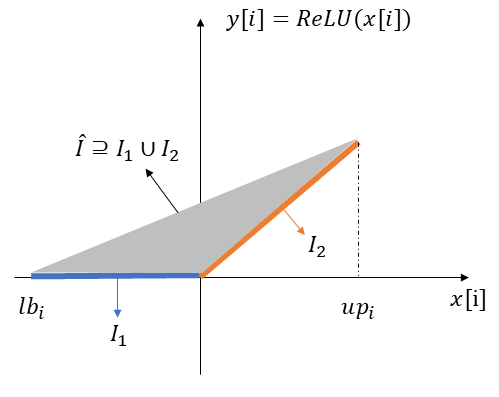}
        \caption{Convex relaxation of the ReLU activation function for over-approximation output analysis.}
        %\caption{Over-approximate output of ReLU activation function.}
        \label{fig:over}
\end{figure}

%\XX{Is reachability analysis and output analysis for FNN the same? A bit confusing. Please check the wording above.}

\begin{algorithm}[!ht]
\caption{Over-approximated output analysis for one layer of FNN}\label{alg:over}
\KwIn{weight matrix $\mathbf W$, bias vector $\mathbf v$, constrained zonotope input set $\mathcal{Z}$}
\KwOut{over-approximated output set $\hat{\mathcal{R}}$}

\SetKwFunction{reachnn}{OverReachNN}
\SetKwFunction{reachrelu}{ReachReLU}
\SetKwFunction{step}{OverStepReLU}
  
  \SetKwProg{Fn}{Function}{:}{}
  \Fn{$\hat{\mathcal{R}}$ = \reachnn{$\mathcal{Z}$,$\mathbf W$,$\mathbf v$}}{
        $\mathcal{I} = \mathbf W \mathcal{Z} + \mathbf v$\\
        $[lb\quad up]\leftarrow$ range of $\x$ in $\mathcal{I}$\\
        $map = find(lb<0)$\\
        \For {$i$ in $map$} {
        $\mathcal{I} = \step (\mathcal{I},i,lb[i],up[i])$
        }
        %$\hat{\mathcal{R}} = \mathcal{I}$\\
        \KwRet $\hat{\mathcal{R}} = \mathcal{I}$\
  }
  
   \SetKwProg{Fn}{Function}{:}{}
  \Fn{$\hat{\mathcal{I}}$ = \step{${\mathcal{I}}$,$i$,$l_i$,$u_i$}}{
        $\mathcal{I} = CZ\{\cc,\G,\A,\bb\} \subset \mathbb{R}^{n_I}$\\ % \{\cc + \G \bm{\xi} | \bm{\xi}\in B_{\infty}(\A,\bb)\}$\\
        $\mathbf E_i = [\mathbf e_1\; \cdots\; \mathbf e_{i-1}\; \mathbf{0}\; \mathbf e_{i+1}\; \cdots\; \mathbf e_{n_I}]$\\
        \If{$u_i \leq 0$}{
          $\hat{\mathcal{I}} = \mathbf E_i {\mathcal{I}}$
         }
        \If{$l_i<0\; \&\; u_i >0$}{
          $\hat{\cc} = \mathbf E_i \cc$\\
          $\hat{\G} = \mat{\mathbf E_i\G & u_i \mathbf{e}_i & \mathbf{0}_{n_I\times 1} & \mathbf{0}_{n_I\times 1} & \mathbf{0}_{n_I\times 1}}$\\
         $\hat{\A} = \mat{\A & \mathbf{0}_{n_A\times 1} & \mathbf{0}_{n_A\times 1} & \mathbf{0}_{n_A\times 1} & \mathbf{0}_{n_A\times 1}\\ \mathbf{0} _{1\times n_{G}} & 1 & 1 & 0 & 0 \\ -\G[i,:] & u_i & 0 & u_i-l_i & 0\\ -\G[i,:] & u_i-l_i & 0 & 0 & u_i-l_i}$\\
          $\hat{\bb} = \mat{\bb & 1 & \cc[i]+u_i-l_i & \cc[i]-u_i }^{T}$\\
          $\hat{\mathcal{I}} = CZ\{\hat \cc,\hat\G,\hat\A,\hat\bb\}$ % = \{\cc + \G \mat{\bm{\xi}\\ \xi_1 \\ \xi_2 \\ \xi_3\\ \xi_4} | \mat{\bm{\xi}\\ \xi_1 \\ \xi_2 \\ \xi_3\\ \xi_4}\in B_{\infty}(\hat \A,\hat\bb)\}
        }
        
        \KwRet $\hat{\mathcal{I}}$\
  }

\end{algorithm}

%\begin{figure}[!ht]
%    \centering
%        \includegraphics[width=0.40\textwidth]{plot-cdc.jpg}
%     \caption{An example of the exact output set and over-approximated output set of a ReLU activation function using Algorithm \ref{alg:exact} and \ref{alg:over}. }
%      \label{fig:relu-output}
%\end{figure}

%\section{Reachability analysis and Safety Verification for Linear Systems}\label{sec:verify}
\section{Reachability analysis and Safety Verification for Neural Feedback System with Linear Model}\label{sec:verify}

In this section, we consider a neural feedback system with a linear model and an FNN controller expressed as follows:
\begin{equation}\label{lin-sys}
    \x{(t+1)} = A_d \x(t) + B_d \pi(\x(t))% + w(t)
\end{equation}
where $A_d \in \mathbb{R}^{n \times n}$ is a given state matrix and other terms are the same as defined in \eqref{dt-sys}.

\subsection{Reachability Analysis} \label{sec:lin}
\subsubsection{Exact Reachability Analysis}
Given an initial set $\mathcal{X}_0$ as a constrained zonotope, Algorithm \ref{alg:exact} and \ref{alg:over} can be used to compute the exact and over-approximated output set of FNN $\pi(\mathcal{X}_0)$ respectively. Now we consider the problem of computing the reachable sets $\mathcal{R}_t(\mathcal{X}_0),\; t=1,2,\dots,T,$ for the closed-loop system \eqref{lin-sys}. Let $\mathcal{R}_0(\mathcal{X}_0) = \mathcal{X}_0$ and $\f_{cl}(\x) = A_d\x + B_d \pi (\x)$. 
A trivial idea is to compute separately the two terms on the right hand side of \eqref{lin-sys} by using Lemma \ref{lemma:set-op} and then take their Minkowski sum; however, the resulting set will be a conservative over-approximation of the true reachable set. 

In the following theorem, we present the exact form of $\f_{cl}(\mathcal{Z})$ for a given constrained zonotope $\mathcal{Z}$.
%, i.e., $\hat{\mathcal{R}}_0 = \mathcal{X}_0$, 
%$\hat{\mathcal{R}}_t = A_d \hat{\mathcal{R}}_{t-1} \oplus B_d \pi(\hat{\mathcal{R}}_{t-1}),\; t=1,2,\dots,T.$ 
%Although the calculated reachable sets $\hat{\mathcal{R}}_{t}$ are over-approximation of the true reachable sets ($\hat{\mathcal{R}}_{t} \supseteq {\mathcal{R}}_{t}$), this method would introduce propagated error which makes the over-approximated reachable sets too conservative. 
%Therefore, we need a better way to combine the two linear terms. 
%For simplicity, we denote $\f_{cl}(\x) = A_d\x + B_d \pi (\x)$.

\begin{theorem}\label{thm:sum}
Given any constrained zonotope $\mathcal{Z} = CZ\{\cc,\G,\A,\bb\} \subset\mathbb{R}^n$, $\G\in\mathbb{R}^{n\times n_G}$, $\A\in \mathbb{R}^{n_A\times n_G}$, let $\pi(\mathcal{Z})$ be the computed output set using Algorithm \ref{alg:exact}, i.e., $\pi(\mathcal{Z}) = \cup_{i=1}^{n_{\pi}} CZ\{\cc_i,\G_i,\A_i,\bb_i\}$. Let $n_{G_i}$ be the number of columns of $\G_i$. Then,% the exact range of $\f_{cl}$ is 
\begin{equation}\label{equ:sum}
    \f_{cl}(\mathcal{Z}) = \cup_{i=1}^{n_{\pi}} CZ\{\cc_i^{cl},\G_i^{cl},\A_i^{cl},\bb_i^{cl}\}
\end{equation}
%$$
%\begin{aligned}
%\f_{cl}(\mathcal{Z}) = \cup_{i=1}^{n_{\pi}} CZ\{A_d\cc + B_d\cc_i,A_d \mat{\G &\mathbf{0}}+B_d\G_i,\A_i,\bb_i\}\\
%\end{aligned}
%$$ 
where %\XX{is it always true that $n_{G_i}>n_{G}$?}
%$
%\G_i^{cl} = A_d \mat{\G &\mathbf{0}_{n \times (n_{G_i}-n_G)}}+B_d\G_i,
%\cc_i^{cl} = A_d\cc + B_d\cc_i,\;\A_i^{cl} = \A_i,\;\bb_i^{cl}=\bb_i.
%$
\begin{align*}
\G_i^{cl} &= A_d \mat{\G &\mathbf{0}_{n \times (n_{G_i}-n_G)}}+B_d\G_i,\\
\cc_i^{cl} &= A_d\cc + B_d\cc_i,\;\A_i^{cl} = \A_i,\;\bb_i^{cl}=\bb_i.
\end{align*}
%$n_G$ is the number of columns of $\G$ and $n_{G_i}$ is the number of columns of $\G_i$.
% $\cc_i^{cl} = A_d\cc + B_d\cc_i$, $\G_i^{cl} = A_d \mat{\G &\mathbf{0}_{n \times (n_{G_i}-n_G)}}+B_d\G_i$, $\A_i^{cl} = \A_i$ and $\bb_i^{cl}=\bb_i$. $n_G$ is the number of columns of $\G$ and $n_{G_i}$ is the number of columns of $\G_i$.
\end{theorem}

\begin{proof}
Denote the right hand side of \eqref{equ:sum} as ${\mathcal{Z}}_R$. Firstly we prove $\f_{cl}(\mathcal{Z}) \subseteq {\mathcal{Z}}_R$. Let $\x$ be any element of set $\mathcal{Z}$, i.e., $\x\in \mathcal{Z}$. We know there exists $\bm{\xi}_1$ such that $\x = \cc+\G\bm{\xi}_1$, $||\bm{\xi}_1||_{\infty} \leq 1$ and $\A\bm{\xi}_1 = \bb$. From the procedures in Algorithm \ref{alg:exact} and Lemma \ref{lemma:halfspace}, it's easy to check that $n_{G_i} \geq n_G$ and the $(n_G+1)$-th to the $n_{G_i}$-th columns of $G_i$ are all zeros. Also, the first $n_A$ rows of $\A_i$ are $\mat{\A & \mathbf{0}_{(n_{G_i}-n_G)\times n_A}}$ and the first $n_A$ rows of $\bb_i$ are $\bb$. 
%Denote the first $n_A$ rows of $\A_i$ as $\A_{i,1}$ and the remaining rows as $\A_{i,2}$. Similarly, denote the first $n_A$ rows of $\bb_i$ as $\bb_{i,1}$ and the remaining rows as $\bb_{i,2}$. 

Since $\pi(\x)\in\pi(\mathcal{Z})$, there must exist $i\in\{1,2,\dots,n_{\pi}\}$ such that $\pi(\x)\in CZ\{\cc_i,\G_i,\A_i,\bb_i\}$. Thus, we know there exists $\bm{\xi} = \mat{\bm{\xi}_1^T & \bm{\xi}_2^T}^T$ such that $\bm{\xi}\in B_{\infty}(\A_i,\bb_i)$ and $\pi(\x) = \cc_i + \G_i\bm{\xi}$. Therefore, $\f_{cl}(\x) = A_d \x + B_d \pi(\x) = A_d (\cc + \G\bm{\xi}_1) + B_d (\cc_i+\G_i \mat{\bm{\xi}_1^T & \bm{\xi}_2^T}^T) = (A_d \cc + B_d\cc_i) + (A_d\mat{\G & \mathbf{0}_{n \times (n_{G_i}-n_G)}} + B_d \G_i) \bm{\xi} = \cc_i^{cl} +\G_i^{cl}\bm{\xi} $, where $\bm{\xi}\in B_{\infty}(\A_i,\bb_i) = B_{\infty}(\A_i^{cl},\bb_i^{cl})$. Then, we have $\f_{cl}(\x) \in CZ\{\cc_i^{cl},\G_i^{cl},\A_i^{cl},\bb_i^{cl}\} \subseteq {\mathcal{Z}}_R$. Since $\x$ is arbitrary, we know that $\f_{cl}(\mathcal{Z}) \subseteq {\mathcal{Z}}_R$.

Next, we show that ${\mathcal{Z}}_R \subseteq \f_{cl}(\mathcal{Z})$. Let $\z\in {\mathcal{Z}}_R$. Then, $\exists i \in \{1,2,\dots,n_{\pi}\}$ such that $\z \in CZ\{\cc_i^{cl},\G_i^{cl},\A_i^{cl},\bb_i^{cl}\}$. Therefore, $\exists \bm{\xi}\in B_{\infty}(\A_i^{cl},\bb_i^{cl})$ such that $\z = \cc_i^{cl} + \G_i^{cl}\bm{\xi} = A_d\cc + B_d\cc_i + (A_d \mat{\G &\mathbf{0}_{n \times (n_{G_i}-n_G)}}+B_d\G_i)\bm{\xi}$. Partitioning $\bm{\xi}$ as $\bm{\xi} = \mat{\bm{\xi}_1^T & \bm{\xi}_2^T}^T$, it follows that $\bm{\xi}_1\in B_{\infty}(\A,\bb)$, $\z = A_d\cc + (A_d \mat{\G &\mathbf{0}_{n \times (n_{G_i}-n_G)}}\mat{\bm{\xi}_1^T & \bm{\xi}_2^T}^T + B_d\cc_i +B_d\G_i \bm{\xi} = A_d\cc + A_d\G\bm{\xi}_1 + B_d\cc_i +B_d\G_i \bm{\xi} $. Let $\x = \cc + \G\bm{\xi}_1$, then $\pi(\x) = \cc_i +\G_i \bm{\xi}$, which implies $\x\in \mathcal{Z}$. Thus, $\z = A_d \x + B_d \pi(\x) = \f_{cl}(\x)\in \f_{cl}(\mathcal{Z})$. Since $\z$ is arbitrary, ${\mathcal{Z}}_R \subseteq \f_{cl}(\mathcal{Z})$. Thus, we conclude that $\f_{cl}(\mathcal{Z}) = {\mathcal{Z}}_R $.
\end{proof}

Theorem 1 can be extended to the case where the input set is a union of constrained zonotopes as shown below.

\begin{corollary}\label{cor:sum}
Given a union of $M$ constrained zonotopes $\mathcal{Z} = \cup_{j=1}^M \mathcal{Z}_j \subset\mathbb{R}^n$, where $\mathcal{Z}_j = CZ\{\cc_j,\G_j,\A_j,\bb_j\}$. Let $\pi(\mathcal{Z}_j) = \cup_{i=1}^{n_{\pi,j}} CZ\{\cc_{i,j},\G_{i,j},\A_{i,j},\bb_{i,j}\}$ be the computed output set using Algorithm \ref{alg:exact} and set $\mathcal{Z}_j$ for $j=1,2,\dots,M$. Then, % the exact range of $\f_{cl}$ is 
%$\f_{cl}(\mathcal{Z}) = \cup_{j=1}^M \f_{cl}(\mathcal{Z}_j) 
%    =  \cup_{j=1}^M \cup_{i=1}^{n_{\pi,j}} CZ\{\cc_{i,j}^{cl},\G_{i,j}^{cl},\A_{i,j}^{cl},\bb_{i,j}^{cl}\},$ 
 \begin{equation}\label{equ:sum2}
 \begin{aligned}
     \f_{cl}(\mathcal{Z}) =& \cup_{j=1}^M \f_{cl}(\mathcal{Z}_j) \\
     =&  \cup_{j=1}^M \cup_{i=1}^{n_{\pi,j}} CZ\{\cc_{i,j}^{cl},\G_{i,j}^{cl},\A_{i,j}^{cl},\bb_{i,j}^{cl}\},
 \end{aligned}
 \end{equation}
where $\A_{i,j}^{cl} = \A_{i,j}$, $\bb_{i,j}^{cl}=\bb_{i,j}$, $\cc_{i,j}^{cl} = A_d\cc_j + B_d\cc_{i,j}$, and $\G_{i,j}^{cl} = A_d [\G_j \;\mathbf{0}_{n \times (n_{G_{i,j}}-n_{G_j})}]+B_d\G_{i,j}$, with $n_{G_j}$ the number of columns of $\G_j$ and $n_{G_{i,j}}$ the number of columns of $\G_{i,j}$.
\end{corollary}

%\begin{remark}
%\XX{Explain high level idea difference; why partition is needed in some work and we don't need it; pros of cons of this proposed method with ``bounding nonlinearity of NN''.}
%\end{remark}

%The proof is similar to Theorem \ref{thm:sum} is therefore is omitted.

Using Theorem \ref{thm:sum} and Corollary \ref{cor:sum}, we can compute 
%the procedure to compute 
the exact reachable sets of closed-loop system \eqref{lin-sys} as follows:%for $T$ time steps can be summarized as follows:
\begin{equation}\label{equ:reach-cl}
    \begin{aligned}
    \mathcal{R}_0 & = \mathcal{X}_0, \\
    \mathcal{R}_t & = \f_{cl}(\mathcal{R}_{t-1}),\; t=1,\dots,T.
    \end{aligned}
\end{equation}

\begin{remark}
In \cite{hu2020reach,everett2021reachability}, over-approximation reachability computation algorithms are proposed for discrete-time systems with FNN controllers. The main idea there is to bound the nonlinearities of FNNs with quadratic or linear constraints. 
%Partitioning the initial set has been shown to affect the conservatism of the convex relaxation involved as local nonlinearities of FNNs fall into smaller ranges of linear or quadratic bounds. 
In contrast, the method proposed in this work includes the FNN nonlinearities in the set-based operations, 
%does not rely on the approximation of FNN nonlinearities, which are included in the set-based operations, 
and therefore, provides a different way for handling reachability analysis of neural feedback systems without bounding or relaxing the FNN nonlinearities. 
%Theorem \ref{thm:sum} shows that the nonlinearities of FNNs can be included in the set-based operations and therefore, exact reachable sets can be computed without any relaxation. 
%partitioning as no relaxation is performed in the process. 
%To the best of our knowledge, Theorem \ref{thm:sum} is the first result that can compute exact reachable sets of a neural feedback system that consists of a linear model and a FNN controller. 
%To the best of knowledge, this is the first exact reachability method for linear feedback systems with FNN controllers. 
%Exact reachability analysis is important for safety-critical systems \cite{ames17cbf}, and 
%such as robots and auto-driving vehicles when interacting with human. Besides, the exact reachable sets 
%can serve as a benchmark to test the conservativeness of other reachability-based methods. 

The price of accuracy, however, is that the number of constrained zonotopes and the order of constrained zonotopes will grow exponentially. Thus, order reduction techniques as proposed in \cite{scott2016constrained,raghuraman2022set} are needed for analyzing deep neural networks. Nevertheless, as shown in Section \ref{sec:sim}, the computation time of our exact analysis algorithm is comparable with other state-of-the-art algorithms.
%\XX{compared other computation methods for reachable sets  of neural feedback systems?  Is this the first paper on exact reachable set of neural feedback system? To the best of our knowledge?}
\end{remark}
\subsubsection{Over-approximation Reachability Analysis}
%The number of constrained zonotopes $n_{\pi}$ in \eqref{equ:sum} depends on the depth and width of the FNN. 
It is computationally demanding to carry out the exact reachability analysis based on \eqref{equ:reach-cl} when $n_\pi$, which depends on the depth and width of the FNN, is large. 
%The computation time for the exact reachability analysis \eqref{equ:reach-cl} will grow exponentially in the worst case. 
The following theorem shows that an over-approximated reachable set can be computed by using Algorithm \ref{alg:over}, to achieve a trade-off between accuracy and efficiency for the reachability analysis.
%form of $\hat \f_{cl}(\mathcal{Z})$ for a given constrained zonotope $\mathcal{Z}$ can be computed.
%we utilize the over-approximated output set computation in Algorithm \ref{alg:over} on the closed-loop system reachability analysis. the exact form of $\f_{cl}(\mathcal{Z})$ for a given constrained zonotope $\mathcal{Z}$. %Similar to exact reachable set computation, we have the following results.
\begin{theorem}\label{thm:appro}
Given any constrained zonotope $\mathcal{Z} = CZ\{\cc,\G,\A,\bb\} \subset\mathbb{R}^n$, $\G\in\mathbb{R}^{n\times n_G}$, $\A\in \mathbb{R}^{n_A\times n_G}$, let $\hat{\pi}(\mathcal{Z})$ be the computed output set using  %over-approximated analysis in 
Algorithm \ref{alg:over}, i.e., $\hat{\pi}(\mathcal{Z}) = CZ\{\hat\cc,\hat\G,\hat\A,\hat\bb\}\supseteq \pi(\mathcal{Z})$. Let $n_{\hat G}$ be the number of columns of $\hat\G$. Then, an over-approximated range of $\f_{cl}(\mathcal{Z})$ can be computed as:
%\XX{Change $\hat R_\pi(\mathcal{Z})$ to $\hat\pi(\mathcal{Z})$}
%$
%    \hat \f_{cl}(\mathcal{Z}) = CZ\{\hat \cc_{cl},\hat\G_{cl},\hat\A_{cl},\hat\bb_{cl}\} \supseteq \f_{cl}(\mathcal{Z}),
%$
\begin{equation}\label{equ:sum3}
    \hat \f_{cl}(\mathcal{Z}) = CZ\{\hat \cc_{cl},\hat\G_{cl},\hat\A_{cl},\hat\bb_{cl}\} \supseteq \f_{cl}(\mathcal{Z})
\end{equation}
where 
%$
%\hat\G_{cl} = A_d \mat{\G &\mathbf{0}_{n \times (n_{\hat G}-n_G)}}+B_d\hat\G, 
%\hat\cc_{cl} = A_d\cc + B_d\hat\cc,\; \hat\A_{cl} = \hat\A,\; \hat\bb_{cl}=\hat\bb. 
%$
\begin{align*}
\hat\G_{cl} &= A_d \mat{\G &\mathbf{0}_{n \times (n_{\hat G}-n_G)}}+B_d\hat\G,\\
\hat\cc_{cl} &= A_d\cc + B_d\hat\cc,\; \hat\A_{cl} = \hat\A,\; \hat\bb_{cl}=\hat\bb. 
\end{align*}
%$\hat\cc_{cl} = A_d\cc + B_d\hat\cc$, $\hat\G_{cl} = A_d \mat{\G &\mathbf{0}_{n \times (n_{\hat G}-n_G)}}+B_d\hat\G$, $\hat\A_{cl} = \hat\A$ and $\hat\bb_{cl}=\hat\bb$. $n_G$ is the number of columns of $\G$ and $n_{\hat G}$ is the number of columns of $\hat\G$.
\end{theorem}

\begin{proof}
From the construction of $\hat\G$ in Algorithm \ref{alg:over} Line 17, we have that $n_{\hat G}\geq n_{G}$. For any $\x \in \mathcal{Z}$, $\exists \bm{\xi}_1$ such that $\x = \cc+\G\bm{\xi}_1$ and $\bm{\xi}_1\in B_{\infty}(\A,\bb)$. 
Since $\pi(\x) \in \pi(\mathcal{Z}) \subseteq \hat{\pi}(\mathcal{Z})$, there exists $\bm{\xi} = \mat{\bm{\xi}_1^T & \bm{\xi}_2^T}^T$ such that $\pi(\x) = \hat\cc+\hat\G\bm{\xi}$ and $\bm{\xi}\in B_{\infty}(\hat\A,\hat\bb)$. Thus, similar to the proof of Theorem \ref{thm:sum}, we have $\f_{cl}(\x) = \hat\cc_{cl} + \hat\G_{cl}\bm{\xi}$ where $\bm{\xi}\in B_{\infty}(\hat\A,\hat\bb)$. Therefore, $\f_{cl}(\x) \in \hat \f_{cl}(\mathcal{Z})$. Since $\x$ is arbitrary, we conclude that $\f_{cl}(\mathcal{Z})\subseteq \hat \f_{cl}(\mathcal{Z})$.
\end{proof}

Using Theorem \ref{thm:appro}, we can compute 
the over-approximated reachable sets of closed-loop system \eqref{lin-sys} as follows:
%The steps to over-approximated the reachable sets of the closed-loop system \eqref{lin-sys} for time horizon $T$ are then given by:
\begin{equation}\label{equ:reach-cl-over}
    \begin{aligned}
    \hat{\mathcal{R}}_0 & = \mathcal{X}_0 \\
    \hat{\mathcal{R}}_t & = \hat \f_{cl}(\hat{\mathcal{R}}_{t-1}),\; t=1,\dots,T.
    \end{aligned}
\end{equation}

\subsection{Safety Verification}
Let the exact reachable set at time $t$ computed by \eqref{equ:reach-cl} be $\mathcal{R}_t(\mathcal{X}_0) = \cup^{n_t}_{i=1}\mathcal{R}_t^i(\mathcal{X}_0) =  \cup^{n_t}_{i=1}CZ\{\cc_i^t,\G_i^t,\A_i^t,\bb_i^t\}$ for $t = 1,2,\dots,T$. Let the $N$ unsafe sets be $\mathcal{O}_j = CZ\{\cc^o_j,\G^o_j,\A^o_j,\bb^o_j\}$ for $j = 1,2,\dots,N$. The following result provides a sufficient and necessary condition on the safety verification of the closed-loop system \eqref{lin-sys}.

\begin{proposition} \label{thm:verify}
Consider the reachable sets $\mathcal{R}_1,\dots,\mathcal{R}_T$ and unsafe sets $\mathcal{O}_1,\dots,\mathcal{O}_N$ defined above, the state trajectories of the closed-loop system \eqref{lin-sys} can avoid all the unsafe regions if and only if the following condition is satisfied:
\begin{align} \label{equ:emp-check}
&\min\{||\bm{\xi}||_{\infty}\;|\; \mat{\A_i^t & \mathbf{0}\\ \mathbf{0} & \A_j^o \\ \G_i^t & -\G_j^o}\bm{\xi} = \mat{\bb_i^t\\ \bb_j^o\\ \cc_j^o - \cc_i^t}\} > 1, \\ \notag
&\forall t \in \{1,\dots,T\},\;\forall i \in \{1,\dots,n_t\},\;\forall j \in \{1,\dots,N\}.
\end{align}
\end{proposition}

%Avoiding the unsafe regions can be equivalently expressed as that none of the reachable sets intersect with any of the unsafe sets. Thus, Proposition \ref{thm:verify} is a straight-forward application of Lemma \ref{lemma:set-op} and Lemma \ref{lemma:empty}.
\begin{proof}
Avoiding the unsafe regions can be equivalently expressed as none of the reachable sets intersect with any of the unsafe sets. According to Lemma \ref{lemma:set-op}, we know the intersection of the $i$-th constrained zonotope from $\mathcal{R}_t$ and the $j$-th unsafe set is also a constrained zonotope, i.e.,
$$
\begin{aligned}
\mathcal{R}_t^i \cap \mathcal{O}_j = CZ\{\cc_i^t,\mat{\G_i^t &\mathbf{0}}, \mat{\A_i^t & \mathbf{0}\\ \mathbf{0} & \A_j^o \\ \G_i^t & -\G_j^o}, \mat{\bb_i^t\\ \bb_j^o\\ \cc_j^o - \cc_i^t}\}.
\end{aligned}
$$
Using Lemma \ref{lemma:empty}, we have that $\mathcal{R}_t^i \cap \mathcal{O}_j$ is empty if and only if \eqref{equ:emp-check} is satisfied. Therefore, the avoidance of all unsafe regions can be certified if and only if \eqref{equ:emp-check} is satisfied for all $t \in \{1,\dots,T\}$, $i \in \{1,\dots,n_t\}$ and $j \in \{1,\dots,N\}$.
\end{proof}

\begin{remark}\label{remark:prop1}
Checking \eqref{equ:emp-check} requires solving $N\sum_{t=1}^T n_t$ LPs with $n_{G^t} + n_{G^O}$ variables and $2(n+n_{G^t} + n_{G^O}+n_{A^t}+n_{A^O})$ constraints. The computation time could increase exponentially with the order of the system and the order of the constrained zonotopes. To reduce the computational burden, order reduction techniques can be employed to limit the complexity of reachable sets by limiting the order of the constrained zonotopes.
%which may prevent the method to be used for real-time applications. 
% A possible solution is to limit the complexity of reachable sets by limiting the order of the constrained zonotopes where, again, order reduction techniques will be employed.
%\XX{comment on LP, its properties, dims? avoidance problem using polytope? interval? ellipsoid? comment on the iff condition, compared with other methods?}
\end{remark}

Given over-approximated reachable sets $\hat{\mathcal{R}}_t(\mathcal{X}_0) =CZ\{\hat \cc^t,\hat\G^t,\hat\A^t,\hat\bb^t\}$ that are computed by \eqref{equ:reach-cl-over}, we have the following result similar to Proposition \ref{thm:verify}.

\begin{proposition}\label{thm:verify2}
The state trajectories of the closed-loop system \eqref{lin-sys} can avoid all the unsafe regions $\mathcal{O}_1,\dots,\mathcal{O}_N$ if 
%The closed-loop system \eqref{lin-sys} is safe if 
%the following condition is satisfied:
\begin{align}\label{equ:emp-check2}
&\min\{||\bm{\xi}||_{\infty}\;|\; \mat{\hat\A^t & \mathbf{0}\\ \mathbf{0} & \A_j^o \\ \hat\G^t & -\G_j^o}\bm{\xi} = \mat{\hat\bb^t\\ \bb_j^o\\ \cc_j^o - \hat\cc^t}\} > 1, \\ \notag
&\forall t \in \{1,\dots,T\},\;\forall j \in \{1,\dots,N\}.
\end{align}
\end{proposition}

%Proof follows the same steps as in Proposition \ref{thm:verify} and is omitted.

Note that there are only $N*T$ LPs in \eqref{equ:emp-check2}, which is a significant reduction than \eqref{equ:emp-check}.
%Compared with \eqref{equ:emp-check}, . Further order reduction can be done for relaxed optimization problems as mentioned in Remark \ref{remark:prop1}.
%Note that since $\hat{\mathcal{R}}_t(\mathcal{X}_0)$ is only an over-approximation of the true reachable set (${\mathcal{R}}_t(\mathcal{X}_0)\subseteq\hat{\mathcal{R}}_t(\mathcal{X}_0)$), condition \eqref{equ:emp-check2} is sufficient but not necessary.
%\begin{remark}
%\XX{comment on the LP above compared with (17)}
%\end{remark}

% \begin{remark}
% Proposition \ref{thm:verify} and Proposition \ref{thm:verify2} presents sufficient conditions on verifying the avoidance of unsafe regions which is a part of the popular reach-avoid problem. Therefore, another remaining problem for neural feedback systems is to check whether the states of the system can reach the goal region in desired time steps. This goal reaching problem can also be converted to a set containment problem by checking whether the final reachable sets are contained in the goal set. \cite{raghuraman2022set} proves the equivalency between constrained zonotopes and affine transformation of polytopes in H-Rep (AH-polytopes). Therefore, the goal reaching problem can be solved by first constrained zonotope sets into AH-polytopes and then apply Theorem 1 of \cite{sadraddini2019linear} for containment verification. More detailed procedures can be found in \cite{raghuraman2022set} which is out of scope of this paper.
% \end{remark}

%\section{Reachability analysis and Safety Verification for Nonlinear Systems} \label{sec:nonlin}
\section{Reachability analysis and Safety Verification for Neural Feedback System with Nonlinear Model} \label{sec:nonlin}

In this section we extend the reachability analysis and safety verification results in the preceding section to the following neural feedback system:
\begin{equation}\label{equ:nonlin}
    \x(t+1) = \f(\x(t)) + B_d \pi(\x(t))
\end{equation}
where $\f$ is assumed to be of class $\mathcal{C}^2$. Let $\f_q$ denote the $q$-th component of function $\f$ and $\mathbf{H} \f_q$ denote the upper triangular matrix describing half of the Hessian of $\f_q$ (i.e. $\mathbf{H}_{ii}\f_q = \frac{\partial^2 \f_q }{2 \partial {\x}_i^2}$, $\mathbf{H}_{ij}\f_q = \frac{\partial^2 \f_q }{\partial {\x}_i \partial \x_j}$ for $i<j$ and $\mathbf{H}_{ij}\f_q = 0$ for $i>j$). Denote $\f_{cl}(\x) = \f(\x) + B_d \pi(\x)$. % \XX{Can we use $\f_{cl}(\x)$ here as above?}

The following proposition provides a method to over-approximate the range of $\f$ using constrained zonotopes.

\begin{proposition}\label{prop:taylor}  \cite[Prop. 2]{rego2021set}
Let $\f: \mathbb{R}^{n} \rightarrow \mathbb{R}^{n}$ be of class $\mathcal{C}^{2}$, and let $\mathcal{X}=CZ\{\mathbf{c},\mathbf{G} , \mathbf{A}, \mathbf{b}\} \subset \mathbb{R}^n$ be a constrained zonotopes with $n_{G}$ generators and $n_{A}$ constraints. For each $q=1,2, \ldots, n$, let $[\mathbf{Q}^{[q]}] \in$ $\mathbb{I R}^{n\times n}$ and $[\tilde{\mathbf{Q}}^{[q]}] \in \mathbb{I R}^{n_{G} \times n_{G}}$ be interval matrices satisfying $[\mathbf{Q}^{[q]}] \supseteq \mathbf{H}_{x} \f_{q}(\Box \mathcal{X} )$ and $[\tilde{\mathbf{Q}}^{[q]}] \supseteq \mathbf{G}^{T} [\mathbf{Q}^{[q]}] \mathbf{G}$. Moreover, define $\tilde{\mathbf{c}}, \tilde{\mathbf{G}}, \tilde{\mathbf{G}}_{\mathbf{d}}, \tilde{\mathbf{A}}$, and $\tilde{\mathbf{b}}$, as in Lemma 2 in \cite{rego2021set}. Finally, choose any $\bm{\gamma}_{x} \in \Box \mathcal{X}$ and let $[\mathbf{L}] \in \mathbb{I R}^{n \times n}$ be an interval matrix satisfying $[\mathbf{L}]_{q,:} \supseteq\left(\mathbf{c}-\gamma_{x}\right)^{T} [\mathbf{Q}^{[q]}]$ for all $q=1, \ldots$, n. Then,
\begin{equation}\label{equ:range}
\f(\mathcal{X}) \subseteq \f \left(\bm{\gamma}_{x}\right) \oplus \nabla_{x}^{T} \f\left(\bm{\gamma}_{x}\right)\left(\mathcal{X}-\bm{\gamma}_{x}\right) \oplus \mathcal{R}  
\end{equation}
where $\mathcal{R}=\tilde{\mathbf{c}} \oplus\left[\tilde{\mathbf{G}}\;\; \tilde{\mathbf{G}}_{\mathbf{d}}\right] B_{\infty}(\tilde{\mathbf{A}}, \tilde{\mathbf{b}}) \oplus \triangleleft ([\mathbf{L}],\left(\mathbf{c}-\bm{\gamma}_{x}\right)$ $\oplus 2 \mathbf{G} B_{\infty}(\mathbf{A}, \mathbf{b}) ) .$
\end{proposition}

Let $\hat\f(\mathcal{X}) = \f \left(\bm{\gamma}_{x}\right) \oplus \nabla_{x}^{T} \f\left(\bm{\gamma}_{x}\right)\left(\mathcal{X}-\bm{\gamma}_{x}\right) \oplus \mathcal{R}$. Since $\mathcal{R}$ defined in Proposition \ref{prop:taylor} is a constrained zonotope, let $\mathcal{R} = CZ\{\cc_R,\G_R,\A_R,\bb_R\}$. Using the set operations of constrained zonotopes in Lemma \ref{lemma:set-op}, we can get
\begin{align} 
    \hat\f(\mathcal{X}) =  & CZ\{\nabla_{x}^{T} \f\left(\bm{\gamma}_{x}\right) (\cc - \bm{\gamma}_{x}) + \f \left(\bm{\gamma}_{x}\right), \nabla_{x}^{T} \f\left(\bm{\gamma}_{x}\right) \G , \nonumber\\ & \A,\bb\} \oplus CZ\{\cc_R,\G_R,\A_R,\bb_R\} \nonumber\\
    = & CZ\{\cc_f,\G_f,\A_f,\bb_f\}\label{equ:hat-f}
\end{align}
where %$\cc_f = \nabla_{x}^{T} \f\left(\bm{\gamma}_{x}\right) (\cc - \bm{\gamma}_{x}) + \f \left(\bm{\gamma}_{x}\right) + \cc_R$, $\G_f = \mat{\nabla_{x}^{T} \f\left(\bm{\gamma}_{x}\right) \G & \G_R}$, $\A_f = \mat{\A & \mathbf{0}\\ \mathbf{0} & \A_R}$ and $\bb_f = \mat{\bb\\ \bb_R}$.
\begin{align*}
    \cc_f & = \nabla_{x}^{T} \f\left(\bm{\gamma}_{x}\right) (\cc - \bm{\gamma}_{x}) + \f \left(\bm{\gamma}_{x}\right) + \cc_R , \bb_f = \mat{\bb\\ \bb_R}, \\ \G_f &= \mat{\nabla_{x}^{T} \f\left(\bm{\gamma}_{x}\right) \G & \G_R}, \A_f = \mat{\A & \mathbf{0}\\ \mathbf{0} & \A_R}.
\end{align*}

%Similarly to the linear system case, the following theorem provides a method to over-approximate the range of $\g_{cl}$.

\begin{theorem}\label{thm:nonlin}
Given any constrained zonotope $\mathcal{Z} = CZ\{\cc,\G,\A,\bb\} \subset\mathbb{R}^n$, $\G\in\mathbb{R}^{n\times n_G}$, $\A\in \mathbb{R}^{n_A\times n_G}$, let $\pi(\mathcal{Z})$ be the computed output set using Algorithm \ref{alg:exact}, i.e., $\pi(\mathcal{Z}) = \cup_{i=1}^{n_{\pi}} CZ\{\cc_i,\G_i,\A_i,\bb_i\} \subset \mathbb{R}^m$. Let $\hat \f(\mathcal{Z}) = CZ\{\cc_f,\G_f,\A_f,\bb_f\}$ be computed as in \eqref{equ:hat-f}. Let $n_{G_i}$ be the number of columns of $\G_i$ and $n_{G_f}$ be the number of columns of $\G_f$. Then, an over-approximated range of $\f_{cl}(\mathcal{Z})$ can be computed as:
%$\hat \g_{cl}$ is an over-approximated range of $\g_{cl}$:
\begin{equation}\label{equ:g-cl}
    \hat \f_{cl}(\mathcal{Z}) = \cup_{i=1}^{n_{\pi}} CZ\{\hat\cc_i^{cl},\hat\G_i^{cl},\hat\A_i^{cl},\hat\bb_i^{cl}\}\supseteq \f_{cl}(\mathcal{Z}),
\end{equation}
where 
%$
%\hat\G_i^{cl} = \mat{\nabla_{x}^{T} \f\left(\bm{\gamma}_{x}\right) \G &\mathbf{0}_{n \times (n_{G_i}-n_G)} & \G_R}\quad\quad +B_d \mat{\G_{i} & \mathbf{0}_{m \times (n_{G_f}-n_G)}},
%\hat\cc_i^{cl} = \cc_f + B_d\cc_i,\;\hat\A_i^{cl} = \mat{\A_i & \mathbf{0} \\ \mathbf{0} & \A_R},\;\hat\bb_i^{cl}=\mat{\bb_i\\ \bb_{R}}.
%$
\begin{align*}
\hat\G_i^{cl} &= \mat{\nabla_{x}^{T} \f\left(\bm{\gamma}_{x}\right) \G &\mathbf{0}_{n \times (n_{G_i}-n_G)} & \G_R}\\
&\quad\quad +B_d \mat{\G_{i} & \mathbf{0}_{m \times (n_{G_f}-n_G)}},\\
\hat\cc_i^{cl} &= \cc_f + B_d\cc_i,\;\hat\A_i^{cl} = \mat{\A_i & \mathbf{0} \\ \mathbf{0} & \A_R},\;\hat\bb_i^{cl}=\mat{\bb_i\\ \bb_{R}}.
\end{align*}
%$\hat\cc_i^{cl} = \cc_f + B_d\cc_i$, $\hat\G_i^{cl} = \mat{\nabla_{x}^{T} \f\left(\bm{\gamma}_{x}\right) \G &\mathbf{0}_{n \times (n_{G_i}-n_G)} & \G_R}+B_d \mat{\G_{i} & \mathbf{0}_{m \times (n_{G_f}-n_G)}}$, $\hat\A_i^{cl} = \mat{\A_i & \mathbf{0} \\ \mathbf{0} & \A_R}$ and $\hat\bb_i^{cl}=\mat{\bb_i\\ \bb_{R}}$. 
%$n_G$ is the number of columns of $\G$, $n_{G_i}$ is the number of columns of $\G_i$ and $n_{G_f}$ is the number of columns of $\G_f$. %$\G_{i,1}$ is the left $n_G$ columns of $\G_i$ and $\G_{i,2}$ is the remaining columns of $\G_i$. $\A_{i,1}$ is the top $n_A$ rows of $\A_i$ and $\A_{i,2}$ is the remaining rows of $\A_i$. $\bb_{i,1}$ is the top $n_A$ rows of $\bb_i$ and $\bb_{i,2}$ is the remaining rows of $\bb_i$. 
\end{theorem}

\begin{proof}
%Denote the exact range of $\g_{cl}$ as $\g_{cl}(\mathcal{Z})$, then we only need to prove that $\g_{cl}(\mathcal{Z}) \subseteq \hat \g_{cl}(\mathcal{Z})$.
For any $\x \in \mathcal{Z}$, $\exists \bm{\xi}_1 \in B_{\infty}(\A,\bb)$ such that $\x = \cc+\G\bm{\xi}_1$. Since $\pi(\x) \in \pi(\mathcal{Z})$, there must exist $i\in\{1,2,\dots,n_{\pi}\}$ and $\bm{\xi}_2$ such that $\mat{\bm{\xi}_1^T & \bm{\xi}_2^T}^T\in B_{\infty}(\A_i,\bb_i)$ and $\pi(\x) = \cc_i + \G_i\mat{\bm{\xi}_1^T & \bm{\xi}_2^T}^T$. From $\mat{\bm{\xi}_1^T & \bm{\xi}_2^T}^T\in B_{\infty}(\A_i,\bb_i)$, we have $ \A_i \mat{\bm{\xi}_1^T & \bm{\xi}_2^T}^T = \bb_i$. Similarly, since $\f(\x) \in \f(\mathcal{Z})$, there exists $\bm{\xi}_3$ such that $\f(\x) = \cc_f + \G_f \mat{\bm{\xi}_1^T & \bm{\xi}_3^T}^T$ and $\mat{\bm{\xi}_1^T & \bm{\xi}_3^T}^T \in B_{\infty}(\A_f,\bb_f)$. Using $\mat{\bm{\xi}_1^T & \bm{\xi}_3^T}^T \in B_{\infty}(\A_f,\bb_f)$, we can get 
$$
\A_f \mat{\bm{\xi}_1 \\ \bm{\xi}_3} = \bb_f \Rightarrow \mat{\A & \mathbf{0}\\ \mathbf{0} & \A_R}  \mat{\bm{\xi}_1 \\ \bm{\xi}_3} = \mat{\bb\\ \bb_R} \Rightarrow \A_R \bm{\xi}_3 = \bb_R.
$$

Let $\bm{\xi} = \mat{\bm{\xi}_1^T & \bm{\xi}_2^T & \bm{\xi}_3^T}^T$, then
$$
\begin{aligned}
\hat\A_i^{cl}\bm{\xi} = \mat{\A_i & \mathbf{0} \\ \mathbf{0} & \A_R} \mat{\bm{\xi}_1 \\ \bm{\xi}_2\\ \bm{\xi}_3} = \mat{\A_i  \mat{\bm{\xi}_1 \\ \bm{\xi}_2} \\ \A_R \bm{\xi}_3} = \mat{\bb_i\\\bb_R}.
\end{aligned}
$$

Because $|| \mat{\bm{\xi}_1^T & \bm{\xi}_2^T}^T ||_{\infty} \leq 1$ and $|| \mat{\bm{\xi}_1^T & \bm{\xi}_3^T}^T ||_{\infty} \leq 1$, we have $||\bm{\xi}||_{\infty}\leq 1$. Thus, $\bm{\xi} \in B_{\infty}(\hat\A_i^{cl}, \hat\bb_i^{cl})$. We know 
$$
\begin{aligned}
\f_{cl}(\x)  =& \f(\x) + B_d\pi(\x) \\=& \cc_f + \G_f \mat{\bm{\xi}_1^T & \bm{\xi}_3^T}^T + B_d(\cc_i + \G_i\mat{\bm{\xi}_1^T & \bm{\xi}_2^T}^T) \\=& \cc_f + B_d\cc_i + (\mat{\nabla_{x}^{T} \f\left(\bm{\gamma}_{x}\right) \G &\mathbf{0} & \G_R}\\ & +B_d \mat{\G_{i} & \mathbf{0}})\bm{\xi} \\=& \hat\cc_i^{cl} + \hat\G_i^{cl} \bm{\xi}
\end{aligned}
$$
Therefore, $\f_{cl}(\x) \in CZ\{\hat\cc_i^{cl},\hat\G_i^{cl},\hat\A_i^{cl},\hat\bb_i^{cl}\} \subseteq \hat \f_{cl}(\mathcal{Z})$. Since $\x$ is arbitrary, we get $\f_{cl}(\mathcal{Z}) \subseteq \hat \f_{cl}(\mathcal{Z})$ which completes the proof.
\end{proof}

Using the same formula as in Corollary \ref{cor:sum}, Theorem \ref{thm:nonlin} can be extended to the case where the input set is a union of constrained zonotopes. Given the initial set $\mathcal{X}_0$ as a constrained zonotope, we can compute 
the over-approximated reachable sets of closed-loop system \eqref{equ:nonlin} as follows:
%the procedure to compute over-approximated reachable sets of systems with nonlinear dynamics \eqref{equ:nonlin} for $T$ time steps can be summarized as follows:
\begin{equation}\label{equ:reach-cl-nonlin}
    \begin{aligned}
    \hat{\mathcal{R}}_0^{nl} & = \mathcal{X}_0 \\
    \hat{\mathcal{R}}_t^{nl} & = \hat\f_{cl}(\hat{\mathcal{R}}_{t-1}^{nl}),\; t=1,\dots,T.
    \end{aligned}
\end{equation}
%$\hat{\mathcal{R}}_0^{nl} = \mathcal{X}_0,\; \hat{\mathcal{R}}_t^{nl}  = \hat\f_{cl}(\hat{\mathcal{R}}_{t-1}^{nl}),\; t=1,\dots,T.
%$

Safety verification for system \eqref{equ:nonlin} can be done similar to Proposition \ref{thm:verify}, by formulating LPs to check the emptiness of $\hat{\mathcal{R}}_t^{nl}\cap\mathcal{O}_j$ for $t\in\{1,\dots,T\}$ and $j\in\{1,\dots,N\}$. The details are omitted due to the space limitation.

\begin{remark}
By replacing $\pi(\mathcal{Z})$ in Theorem \ref{thm:nonlin} with an over-approximated set $\hat\pi(\mathcal{Z})$ using Algorithm \ref{alg:over}, we can reduce the computational complexity of getting $\hat \f_{cl}(\mathcal{Z})$. However, in this case, both the linearization error and the FNN over-approximation error will appear in the set propagation. In \cite{sidrane2022overt}, an algorithm is  proposed to abstract nonlinear functions with a set of optimally tight piecewise linear bounds which can be integrated with the set-based method in this work. %\XX{please check.}
%still compute the over-approximated reachable sets for the nonlinear system \eqref{equ:nonlin}. However, this will introduce more conservatism to the results than \eqref{equ:reach-cl-nonlin} as both the linearization error and the FNN over-approximation error will appear in the set propagation.
\end{remark}

% \begin{remark}
% In \cite{everett2021reachability}, a method based on quadratically constrained quadratic programming (QCQP) is proposed to approximate reachable sets for neural feedback systems with polynomial dynamics. However, our work in this section is applicable for more general nonlinear dynamics as long as it satisfies the twice continuously differentiable assumptions.
% \end{remark}

\section{Simulation}\label{sec:sim}
In this section, we demonstrate the performance of the proposed reachability analysis methods using two simulation examples. %We firstly apply results of Section \ref{sec:verify} to a double integrator system with a FNN controller, and compare the efficiency and accuracy of the proposed method with other state-of-the-art methods. We then present a nonlinear system example using Theorem \ref{thm:nonlin} of Section \ref{sec:nonlin}.

%, Reach-CZ and Reach-CZ-Approx where CZ stands for constrained zonotope,

\subsection{Double Integrator Example}

Consider a double integrator model \cite{hu2020reach,everett2021reachability}:
$$
{\x}{(t+1)}=\left[\begin{array}{ll}
1 & 1 \\
0 & 1
\end{array}\right] {\x}(t)+\left[\begin{array}{c}
0.5 \\
1
\end{array}\right] {\u}(t).
$$
The feedback controller is set to be a 3-layer FNN with ReLU activation functions and the same parameters as used in \cite{everett2021reachability}. We implement both Algorithm \ref{alg:exact} and Algorithm \ref{alg:over} to get exact and over-approximated output sets of the FNN and then utilize Corollary \ref{cor:sum} and Theorem \ref{thm:appro} to compute the reachable sets of the closed-loop system for $T = 5$ time steps. The initial set is given by $[2.5,3.0]\times[-0.25,0.25]$.

We denote the proposed exact reachability analysis method based on \eqref{equ:reach-cl} and Theorem \ref{thm:sum} as Reach-CZ and denote the over-approximation reachability analysis method based on \eqref{equ:reach-cl-over} and Theorem \ref{thm:appro} as Reach-CZ-Approx. We use the open-source Python toolboxes \emph{nn\_robustness\_analysis} (\cite{Everettgithub}) to run the Reach-LP algorithm (\cite{everett2021reachability}) and the Reach-SDP algorithm (\cite{hu2020reach}), and the versions with Greedy Sim-Guided Partition (\cite{everett2020robustness}) for the initial sets, i.e., Reach-LP-Partition and Reach-SDP-Partition. All the parameters are kept as default. 
Table \ref{tbl} summarizes the computation times and set over-approximation errors for the proposed method and other state-of-the-art methods including Reach-LP, Reach-LP-Partition, Reach-SDP, and Reach-SDP-Partition. The approximation errors are computed based on the difference ratio of sizes of over-approximated reachable sets and exact reachable sets at the last time step. Note   that the proposed Reach-CZ method can return the \emph{exact} reachable sets within a reasonable time. Note also that the proposed Reach-CZ-Approx method provides a better balance between efficiency and accuracy. Using about half the time consumed by Reach-LP-Partition, Reach-CZ-Approx achieves an approximation error that is over 20 times smaller than Reach-LP-Partition. Our constrained zonotope-based algorithms are implemented in Python with MOSEK \cite{andersen2000mosek}. All algorithms are tested in a computer with 3.7 GHz CPU and 32 GB memory.

Figure \ref{fig:linear-exp} illustrates reachable sets of the double integrator system using different methods.  It can be observed that our method provides more accurate reachable sets for all the time steps compared with  other methods. This is beneficial to avoid false unsafe detection in the safety verification problem; 
%by using reachable sets computed by Reach-CZ or Reach-CZ-Approx; 
for example, with the unsafe region given in Figure \ref{fig:linear-exp}, Reach-CZ and Reach-CZ-Approx can verify the safety of the neural feedback system while other methods can not. 

% Figure \ref{fig:linear-exp} illustrates reachable sets of the double integrator system with our constrained zonotope-based algorithms and algorithms from \cite{hu2020reach,everett2021reachability}. It's obvious that our proposed algorithms provide more accurate reachable sets for all the time steps comparing with the other algorithms. This is beneficial for the safety verification problem as using reachable sets computed by Reach-CZ or Reach-CZ-Approx, false unsafe detection can be prevented.

\begin{table}[!ht]
\centering
\scalebox{1.1}{
\begin{tabular}{|c|c|c|}
\hline
Algorithm                  & Runtime [s] & Approx. Error \\ \hline
Reach-CZ  (ours)           & 1.214 &  0  \\ 
Reach-CZ-Approx (ours) &  0.320  &  0.8 \\ \hline
Reach-LP \cite{everett2021reachability}                   &   0.031      &     330                \\ 
Reach-LP-Partition         &    0.891     &         19            \\ 
Reach-SDP \cite{hu2020reach}       & 56.03     &     207   \\ 
Reach-SDP-Partition        &   2048.89  &   11   \\ \hline
\end{tabular}}
\caption{Comparison of different reachability-based methods for the double integrator example. Reach-CZ returns exact reachable sets within a reasonable time. Compared with Reach-LP-Partition, Reach-CZ-Approx achieves over 20 times smaller errors using about half its time. }
\label{tbl}
\end{table}
%\vspace{-0.3cm}

\begin{figure}[!ht]
    \centering
        \includegraphics[width=0.43\textwidth]{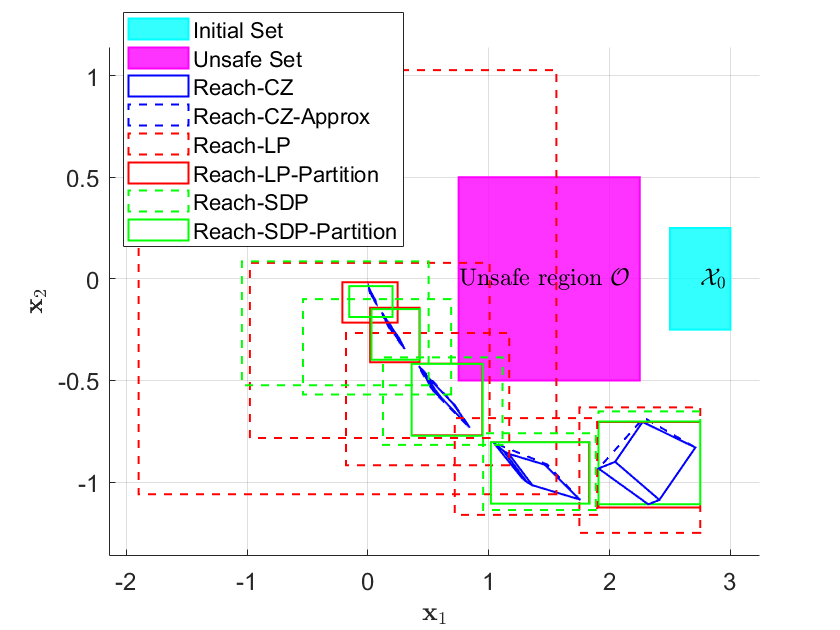}
     \caption{Reachable sets computed for the double integrator example. The initial set $\mathcal{X}_0$ is shown in cyan and the unsafe region represented by set $\mathcal{O}$ is in magenta. Reachable sets computed by Reach-CZ and Reach-CZ-Approx are bounded by blue solid lines and blue dashed lines respectively. Reachable sets computed by the LP-based method (\cite{everett2021reachability}) is in red while the SDP-based method (\cite{hu2020reach}) is in green.}
      \label{fig:linear-exp}
      %\vspace{-0.2cm}
\end{figure}

\subsection{Nonlinear System Example}
%To demonstrate the proposed method on system with nonlinear dynamics, we 
Consider the following discrete-time Duffing Oscillator model from \cite{dang2012reachability}:
$$
\begin{aligned}
x_1(t+1) &= x_1(t)+0.3 x_2(t)\\
x_2(t+1) &= 0.3 x_1(t) + 0.82 x_2(t) -0.3 [x_1(t)]^3 + 0.3 u(t)
\end{aligned}
$$
where $x_1,x_2 \in \mathbb{R}$ are the states and $u\in \mathbb{R}$ is the control input. We train an FNN controller to approximate the control law described in \cite{dang2012reachability} and apply the algorithm based on Theorem \ref{thm:nonlin} for $T = 2$ time steps. Figure \ref{fig:nonlinear-exp} shows the computed reachable sets and 1000 randomly generated sample trajectories. The sampled states are contained in the over-approximated reachable sets as expected. For comparison, we also apply the method based on quadratically constrained quadratic programming (QCQP) in \cite{everett2021reachability}, denoted as Reach-QCQP, to approximate the reachable sets for the Duffing Oscillator neural feedback system. Figure \ref{fig:nonlinear-exp} indicates that although all the three methods only over-approximate the reachable sets, our constrained zonotope-based methods tend to have tiger bounds on the sampled states.

\begin{figure}[!t]
    \centering
        \includegraphics[width=0.40\textwidth]{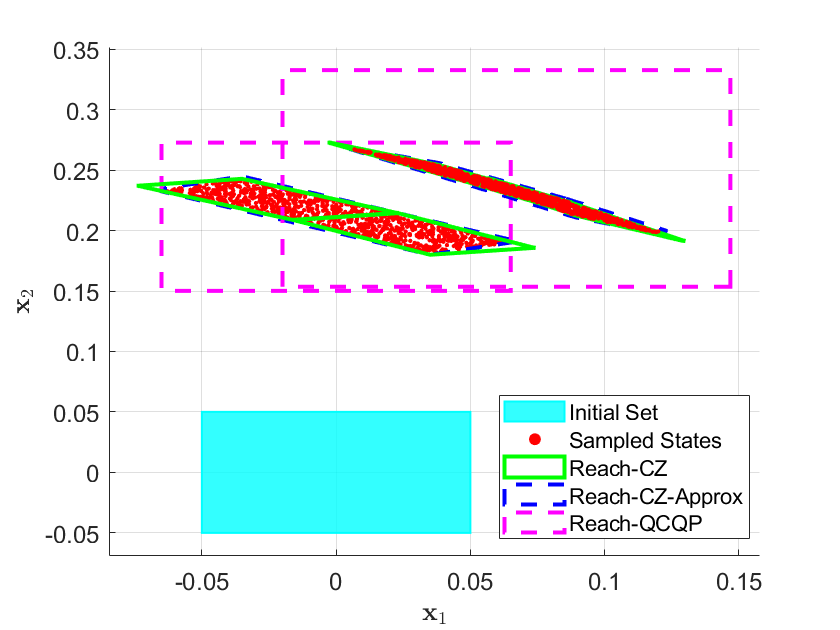}
     \caption{Reachable sets computed for the Duffing Oscillator system. Initial set $\mathcal{X}_0$ is shown in cyan. Reachable sets computed by Reach-CZ and Reach-CZ-Approx are bounded by green solid lines and blue dashed lines respectively. Sampled states from 1000 randomly generated initial conditions are plotted as red dots which are bounded by over-approximated reachable sets. Reachable sets computed by the QCQP-based method in \cite{everett2021reachability} are shown by purple dashed lines. }
      \label{fig:nonlinear-exp}
      %\vspace{-0.2cm}
\end{figure}

\section{Conclusion}\label{sec:concl}
In this paper, we proposed a constrained zonotope-based method for analyzing the exact and over-approximated reachable sets of neural feedback systems. The exact reachable set of the neural feedback system is a union of constrained zonotopes and can be computed in a reasonable amount of time. The over-approximated method has much higher time efficiency than the exact method with a slight loss of accuracy. Based on the reachability analysis, we provided two LP-based conditions for safety verification of the neural feedback system. We also extended the proposed methods to a class of nonlinear systems. For future work, we %plan to extend the reachability analysis to neural networks with other types of activation functions such as sigmoid and hyperbolic tangent and to more general nonlinear systems. Besides, we also 
plan to explore the tunability of constrained zonotopes to achieve a better trade-off between computational efficiency and  approximating accuracy.

%\XX{Double check the references. E.g., Capital letter (e.g. in [1],[2]), redundant year number (e.g. [7])}

\bibliographystyle{IEEEtran}
\bibliography{ref}

\end{document}